\documentclass[11 pt]{article}
\usepackage{amsmath, amsthm,amssymb,bbm,enumerate,mathrsfs,mathtools}
\usepackage{a4wide}
\usepackage{tikz,xcolor,verbatim}
\usetikzlibrary{calc,shapes}
\usepackage{hyperref}
\usepackage{makecell}
\usepackage[numbers,sort&compress]{natbib}
\usepackage{comment}
\usepackage{amssymb}
\usepackage{algorithmicx}
\usepackage{algpseudocode}
\usepackage{algorithm}
\algblock{Input}{EndInput}
\algnotext{EndInput}
\algrenewtext{Input}{\textbf{Input: }}
\algblock{Output}{EndOutput}
\algnotext{EndOutput}
\algrenewtext{Output}{\textbf{Output: }}
\algblock{Initialization}{EndInitialization}
\algnotext{EndInitialization}
\algblock{Begin}{EndBegin}
\algnotext{EndBegin}
\algblock{Recurse}{EndRecurse}
\algnotext{EndRecurse}
\algrenewtext{Recurse}{\textbf{recurse }}
\algblock{Params}{EndParams}
\algnotext{EndParams}
\algrenewtext{Params}{\textbf{Parameters: }}

\newcommand{\abs}[1]{\left|#1\right|}

\tikzset{
  bigblue/.style={circle, draw=blue!80,fill=blue!40,thick, inner sep=1.5pt, minimum size=5mm},
  bigred/.style={circle, draw=red!80,fill=red!40,thick, inner sep=1.5pt, minimum size=5mm},
  bigblack/.style={circle, draw=black!100,fill=black!40,thick, inner sep=1.5pt, minimum size=5mm},
  bluevertex/.style={circle, draw=blue!100,fill=blue!100,thick, inner sep=0pt, minimum size=2mm},
  redvertex/.style={circle, draw=red!100,fill=red!100,thick, inner sep=0pt, minimum size=2mm},
  blackvertex/.style={circle, draw=black!100,fill=black!100,thick, inner sep=0pt, minimum size=1.5mm},  
  whitevertex/.style={circle, draw=black!100,fill=white!100,thick, inner sep=0pt, minimum size=2mm},  
  smallblack/.style={circle, draw=black!100,fill=black!100,thick, inner sep=0pt, minimum size=1mm},
  smallwhite/.style={circle, draw=black!100,fill=white!100,thick, inner sep=0pt, minimum size=1mm} 
}

\makeatletter
\pgfdeclareshape{myNode}{
  \inheritsavedanchors[from=rectangle] 
  \inheritanchorborder[from=rectangle]
  \inheritanchor[from=rectangle]{center}
  \inheritanchor[from=rectangle]{north}
  \inheritanchor[from=rectangle]{south}
  \inheritanchor[from=rectangle]{west}
  \inheritanchor[from=rectangle]{east}
  \backgroundpath{
    \southwest \pgf@xa=\pgf@x \pgf@ya=\pgf@y
    \northeast \pgf@xb=\pgf@x \pgf@yb=\pgf@y
    \pgfsetcornersarced{\pgfpoint{5pt}{5pt}}
    \pgfpathmoveto{\pgfpoint{\pgf@xa}{\pgf@ya}}
    \pgfpathlineto{\pgfpoint{\pgf@xa}{\pgf@yb}}
    \pgfpathlineto{\pgfpoint{\pgf@xb}{\pgf@yb}}
    \pgfsetcornersarced{\pgfpoint{5pt}{5pt}}
    \pgfpathlineto{\pgfpoint{\pgf@xb}{\pgf@ya}}
    \pgfpathclose
 }
}
\makeatother

\newenvironment{claimproof}[1]{\par\noindent\underline{Proof:}\space#1}{\leavevmode\unskip\penalty9999 \hbox{}\nobreak\hfill\quad\hbox{$\blacksquare$}}

\usepackage{xcolor}

\usepackage[margin=1in]{geometry}

\title{Local Hadwiger's Conjecture}

\author{Benjamin Moore
\thanks{Computer Science Institute, Charles University, Prague, Czech Republic, {\tt brmoore@iuuk.mff.cuni.cz}.}
\and
Luke Postle
\thanks{Combinatorics and Optimization Department,
University of Waterloo, Waterloo, Ontario N2L 3G1, Canada {\tt lpostle@uwaterloo.ca}. Partially supported by NSERC
under Discovery Grant No. 2019-04304 and the Canada Research Chairs program.}
\and
Lise Turner
\thanks{Department of Combinatorics and Optimization, University of Waterloo, Waterloo, ON, Canada {\tt lise.turner@uwaterloo.ca}.}}

\newtheorem{thm}[equation]{Theorem}
\newtheorem{lemma}[equation]{Lemma}

\newtheorem{conj}[equation]{Conjecture}
\newtheorem{cor}[equation]{Corollary}
\newtheorem{claim}{Claim}

\newtheorem{obs}[equation]{Observation}

\theoremstyle{definition}
\newtheorem{definition}[equation]{Definition}

\newtheoremstyle{case}{}{}{\normalfont}{}{\itshape}{\normalfont:}{ }{}

\theoremstyle{case}

\allowdisplaybreaks

\numberwithin{equation}{section}

\date{}

\begin{document}

\maketitle

\begin{abstract}
We propose local versions of Hadwiger's Conjecture, where only balls of radius $\Omega(\log(v(G)))$ around each vertex are required to be $K_{t}$-minor-free. We ask: if a graph is locally-$K_{t}$-minor-free, is it $t$-colourable? We show that the answer is yes when $t \leq 5$, even in the stronger setting of list-colouring, and we complement this result with a $O(\log v(G))$-round distributed colouring algorithm in the LOCAL model. Further, we show that for large enough values of $t$, we can list-colour locally-$K_{t}$-minor-free graphs with $13\cdot \max\left\{h(t),\left\lceil \frac{31}{2}(t-1) \right\rceil \right\})$colours, where $h(t)$ is any value such that all $K_{t}$-minor-free graphs are $h(t)$-list-colourable. We again complement this with a $O(\log v(G))$-round distributed algorithm. 
\end{abstract}

\section{Introduction}
Hadwiger's Conjecture is one of the most famous open problems in graph theory and is a vast generalization of the Four Colour Theorem. The conjecture involves two storied areas of graph theory: graph colouring and graph minors. A graph $G$ has a \textit{$k$-colouring} if there is a map $f:V(G) \to \{1,\ldots,k\}$ such that for every edge $e=xy$, we have $f(x) \neq f(y)$. A graph $G$ has an \emph{$H$ minor} if a graph isomorphic to $H$ can be obtained from a subgraph of $G$ by contracting edges.

Hadwiger's Conjecture asserts a tight bound for the chromatic number of graphs not containing a large $K_{t}$ minor. 
\begin{conj}[Hadwiger's Conjecture \cite{Hadwigerconj}]
Every graph with no $K_{t}$ minor is $(t-1)$-colourable.
\end{conj}

In this paper, we propose local variants of Hadwiger's Conjecture wherein a graph is only `locally' $K_t$-minor-free; such versions are also necessary for the existence of LOCAL distributed algorithms for colouring $K_t$-minor-free graphs. We then prove various local versions (Theorems~\ref{smallvaluesoft} and~\ref{largevaluesoft}) and provide the complementary local algorithms (Theorem~\ref{algorithmicresults}); in particular, we show an asymptotic equivalence between the non-local version and the local version (and even the existence of a local distributed algorithm). 

First, we survey some of the known results on Hadwiger's Conjecture as follows. Hadwiger's Conjecture is rather straightforward to prove when $t \leq 4$. When $t = 5$, due to a structural characterization of $K_{5}$-minor-free graphs by Wagner~\cite{Wagnersthm}, the conjecture is equivalent to the famous Four Colour Theorem proved by Appel and Haken~\cite{fourcolourtheorem} in the 1970s. The $t=6$ case was proved by Robertson, Seymour and Thomas~\cite{Hadwiger6} in 1993, while for all $t \geq 7$, the conjecture remains open. Due to the difficulty of Hadwiger's Conjecture, the following weakening has received significant attention.

\begin{conj}[Linear Hadwiger's Conjecture \cite{LinearHadwiger}]
There exists a constant $c >0$ such that every graph with no $K_{t}$ minor is $ct$-colourable.
\end{conj}

Since this conjecture is open, a natural question is: what is the best function $f(t)$ such that for all graphs $G$ with no $K_{t}$ minor, $G$ is $f(t)$-colourable? For a long time, the best known function $f(t)$ was $O(t\sqrt{\log(t)})$, proven independently by Kostochka \cite{degeneracyKostochka1, degeneracyKostochka2} and Thomason \cite{thomason}. In particular, they proved a bound on the degeneracy of $K_{t}$-minor-free graphs. Recall, for an integer $d$, a graph $G$ is \textit{$d$-degenerate} if every subgraph $H$ of $G$ contains a vertex of degree at most $d$. Independently, Kostochka and Thomason proved the following result.

\begin{thm}[\cite{degeneracyKostochka1,degeneracyKostochka2,thomason}]
Every graph with no $K_{t}$ minor is $O(t\sqrt{\log(t)})$-degenerate.
\end{thm}

It is known that there exist graphs with no $K_{t}$ minor where the the degeneracy bound above is tight \cite{thomason} and hence the above result is best possible. In fact, in \cite{REED1998147}, it is mentioned that a commonly expressed ``counter" conjecture was that $f(t)$ was $\Omega(t\sqrt{\log(t)})$. Recently Norin, the second author, and Song \cite{degenbroken} disproved this assertion as follows.

\begin{thm}[\cite{degenbroken}]
For every $\beta >\frac{1}{4}$, every graph with no $K_{t}$ minor is $O(t\log(t)^{\beta})$-colourable.
\end{thm}

The best current bound is due to Delcourt and the second author~\cite{Smallgraphs}, who proved the following result.

\begin{thm}[\cite{Smallgraphs}]
Every graph with no $K_{t}$ minor is $O(t\log\log(t))$-colourable. 
\end{thm}

Furthermore, Delcourt and the second author~\cite{Smallgraphs} also showed that to prove linear Hadwiger's conjecture, it suffices to prove the conjecture for small graphs, that is those with at most $O(t\log^4 t)$ vertices. 

Despite the fact that Linear Hadwiger's Conjecture remains open, researchers have already proposed various strengthenings of the conjecture. Recall that for a graph $G$, a \textit{list assignment} $L$ of $G$ is a collection of sets $(L(v): v \in V(G))$. An \textit{$L$-colouring} of a graph $G$ is a colouring $f$ of $G$ such that $f(v) \in L(v)$ for all $v \in V(G)$. If a graph $G$ has an $L$-colouring for every list assignment $L$ where $|L(u)| \geq k$ for every $u \in V(G)$, then $G$ is \textit{$k$-list-colourable}. One particularly nice conjecture is as follows.

\begin{conj}[List Linear Hadwiger's Conjecture \cite{listhadwiger}]
There exists a constant $c >0$ such that every graph with no $K_{t}$ minor is $ct$-list-colourable. 
\end{conj}

Some lower bounds on the value of $c$ are known. In 2011, Bar\'at, Joret and Wood~\cite{Listdisproof} constructed a graph with no $K_{3t+2}$ minor which is not $4t$-list-colourable for every integer $t\ge 1$. Quite recently, Steiner~\cite{steiner_2022} provided a lower bound of $(2-o(1))t$. For upper bounds, the following is the current best known result.

\begin{thm}[\cite{furtherprogresslist}]
Every graph with no $K_{t}$ minor is $O(t(\log\log(t))^{6})$-list-colourable. 
\end{thm}

It is important to observe in these theorems and conjectures that having no $K_{t}$-minor is a global condition. It is natural to ask if we can prove similar colouring bounds when only requiring a more local condition. We formalize this notion as follows. Recall that for a graph $G$, a vertex $v$, and a positive integer $c$, we say that the \textit{ball of radius $c$ around $v$}, denoted $B_c[v]$ is the subgraph induced by all vertices at distance at most $c$ from $v$.

\begin{definition}
Let $c$ be a positive integer. We say a graph $G$ is \textit{$c$-locally-$K_{t}$-minor-free} if for every vertex $v \in V(G)$, we  have that $B_c[v]$ is $K_{t}$-minor-free.
\end{definition}

Throughout we use the notation $v(G)$ to denote the number of vertices of $G$, and $e(G)$ to denote the number of edges of $G$. We propose two rather strong generalizations of Hadwiger's conjecture:

\begin{conj}[Local Hadwiger's Conjecture]
\label{localhadwigerconjecture}
Let $t$ be a positive integer. There exists a constant $c_{t}>0$ depending on $t$ such that every graph $G$ which is $\lceil c_{t}\log(v(G)) \rceil$-locally-$K_{t}$-minor-free is $t$-colourable.
\end{conj}

\begin{conj}[Local List Linear Hadwiger's conjecture]
Let $t$ be a positive integer. There exists a constant $c_{1} >0$ depending on $t$, and a constant $c_{2} >0$ such that every graph $G$ which is $\lceil c_{1}\log(v(G)) \rceil$-locally-$K_{t}$-minor-free is $c_{2}t$-list-colourable. 
\end{conj}

Note one cannot replace $t$ with $t-1$ in Conjecture \ref{localhadwigerconjecture}. In particular when $t=2$, the odd cycle $C_{2k+1}$ is $(k-1)$-locally-$K_{3}$-minor free and yet odd cycles are not $2$-colourable. More generally, in Section~\ref{lowerboundssection}, we provide a construction of $\Omega(v(G))$-locally-$K_t$-minor-free graphs which are not $(t-1)$-colourable for all integers $t\ge 2$, thus showing that Conjecture \ref{localhadwigerconjecture} is in fact best possible for every $t$. Additionally, one needs the locality condition to depend on $v(G)$, due to the classic result of Erd\H{o}s which says there are graphs of large girth and large chromatic number (see for instance, \cite{diestel}).  

Our first result is that Conjecture~\ref{localhadwigerconjecture} is true when $t \leq 5$, even in the setting of list colouring. 

\begin{thm}
\label{smallvaluesoft}
For each $t \in \{3,4,5\}$, there exists a positive integer $c_{t}$ such that if $G$ is a graph which is $\lceil c_{t}\log(v(G)) \rceil$-locally-$K_{t}$-minor-free, then $G$ is $t$-list-colourable.
\end{thm}

For large values of $t$, we show that if $h(t)$ is the smallest value such that all graphs with no $K_{t}$ minor are $h(t)$-list-colourable, then either $13h(t)$-suffices in the local setting or a linear function in $t$ suffices. More precisely, we prove the following. 

\begin{thm}
\label{largevaluesoft}
Let $t$ be a positive integer. There exists a constant $c_{t} >0$ depending on $t$ such that if $G$ is $\lceil c_{t}\log(v(G)) \rceil$-locally-$K_{t}$-minor-free, then $G$ is $13k$-list-colourable, where $k = \max\left\{h(t),\left\lceil \frac{31}{2}(t-1) \right\rceil \right\}$ and $h(t)$ is the smallest value such that all $K_t$-minor-free graphs are $h(t)$-list-colourable. 
\end{thm}
 
As mentioned earlier, we also give a construction showing $t$ would be tight in Conjecture~\ref{localhadwigerconjecture} as follows.

\begin{obs}
\label{construction}
For each integer $t\ge 3$ and positive integer $k$, there exists a graph $G$ on $k(t-1)+1$ vertices which is $\lfloor k/2 \rfloor$-locally-$K_{t}$-minor-free but not $(t-1)$-colourable.
\end{obs}

To prove these theorems, we utilize the theory of hyperbolic families of graphs developed by the second author and Thomas~\cite{Hyperbolicfamilies}. This theory has been used to great success for colouring graphs on surfaces, see~\cite{Hyperbolicfamilies}. To utilize this theory, we need to prove certain precolouring extension type results for graphs with no $K_{t}$ minor. 

We need some definitions to properly state these results. First we define a notion of criticality with respect to a list assignment and a precoloured set.

\begin{definition}
Let $G$ be a graph, $L$ a list assignment of $G$, and $X$ a subgraph of $G$. We say $G$ is \emph{$X$-critical with respect to $L$}, or \emph{$(X,L)$-critical}, if for every proper subgraph $H$ of $G$ containing $X$, there exists an $L$-colouring of $X$ that extends to $H$ but not to $G$.
\end{definition}

We recall the second author's definition of deletable subgraphs from~\cite{Postlealgorithms}. Note for a vertex $v$ in a graph $G$, we use $d_{G}(v)$ to denote the \textit{degree} of $v$ in $G$. 

\begin{definition}
Let $G$ be a graph, $H$ be a non-empty induced subgraph of $G$, and $r$ a positive integer. We say that $H$ is \emph{$r$-deletable} if for every list assignment $L$ of $H$ such that all $v\in V(H)$ satsify $\abs{L(v)}\geq r-(d_G(v)-d_H(v))$, we have that $H$ is $L$-colourable.
\end{definition}

This is an extremely strong definition of deletability. Indeed, the definition immediately implies that, if $H$ is $r$-deletable, then for any $r$-list assignment $L$ of $G$, any $L$-colouring of $G-H$ extends to an $L$-colouring of $G$. In particular, that implies that if $G$ is $(X,L)$-critical for some subgraph $X$ of $G$, then $G$ does not contain an induced $r$-deletable subgraph disjoint from $X$. We propose a rather strong conjecture as follows.

\begin{conj}
\label{deletablesubgraphconjecture}
There exists a constant $k$ such that for every $t$, there is a constant $c_{t}$ depending on $t$ such that if $G$ is a $K_{t}$-minor-free graph, $X\subseteq V(G)$ and $G$ has no $kt$-deletable subgraph disjoint from $X$, then $v(G) \leq c_{t}|X|$.
\end{conj}

We prove this for small values of $t$ with a best possible value of $k$.

\begin{thm}
Conjecture \ref{deletablesubgraphconjecture} is true with $k=1$ when $t \leq 5$.
\end{thm}

We prove a weaker version of Conjecture \ref{deletablesubgraphconjecture} when $t$ is large by showing there exists no $t$-deletable subgraph where $t$ depends on the bounds known for Linear List Hadwiger's Conjecture. With these results combined with some ideas from the theory of hyperbolic families, we deduce Theorem \ref{smallvaluesoft} and Theorem \ref{largevaluesoft}.

While we have not mentioned it so far, this work was motivated by distributed algorithms. For large scale networks, it can be very useful and more efficient to have many different computers running in parallel to achieve a single goal. There are many different models for distributed computing, and for this work we use the LOCAL model of computation. In the LOCAL model, every vertex has its own processor, computation on these processors is free, and time is measured in \emph{rounds}. Every round, vertices can send messages of unlimited size to their neighbours. The \emph{runtime} of an algorithm is defined as the number of rounds required to complete the computation, that is to have every vertex output its part of the solution (e.g.~its colour in a colouring). \emph{Efficient} deterministic algorithms are those whose run time is polylogarithmic in the number of vertices. Indeed, this is natural, since in $v(G)$ rounds, each vertex will have complete information on its component which when combined with unlimited processing power is sufficient to brute force a solution. 

In this paper, we ask: Given a graph $G$ with no $K_{t}$ minor, and a list assignment $L$, how many rounds does a distributed algorithm require in the LOCAL model to $L$-colour $G$ (assuming such a colouring exists)? Hence we are asking for an algorithmic version of the above existence results. 

When $t=3$, note that $K_{3}$-minor-free graphs are simply trees. For this case, distributed algorithms were given by Cole and Vishkin~\cite{CV1986} and by Goldberg, Plotkin and Shannon~\cite{GoldbergPlotkin} which $3$-colour a tree $G$ in $O(\log^*(v(G)))$ time (recall that $\log^{*}$ denotes the number of times the logarithm function must be applied before the value is below one). However, to the best of the authors' knowledge, for all larger values of $t$ there is no literature on this topic. 

However, for graphs embeddable in some fixed surface, much work has been done. Goldberg, Plotkin and Shannon~\cite{GoldbergPlotkin} gave a deterministic distributed algorithm for $7$-colouring a planar graph $G$ in $O(\log(v(G)))$ rounds. In 2019, Aboulker, Bonamy, Bousquet, and Esperet~\cite{aboulker2018distributed} gave a deterministic distributed algorithm for $6$-list-colouring a planar graph $G$ in $O(\log^3(v(G)))$ rounds. The second author~\cite{Postlealgorithms}, using the theory of hyperbolic families, gave a deterministic distributed algorithm for every fixed surface $\Sigma$ to $5$-list-colour a graph $G$ embedded in $\Sigma$ in $O(\log(v(G)))$ rounds. 

For convenience, we say a list assignment $L$ of a graph $G$ is \emph{type $345$} if $L$ is a $k$-list-assignment for some $k \in \{3,4,5\}$ and $G$ has girth at least $8-k$. Here is the second author's result:

\begin{thm}[\cite{Postlealgorithms}]
\label{lukesalgorithmtheorem}
For each surface $\Sigma$, there exists a deterministic distributed algorithm in the LOCAL model that, given a
graph $G$ embeddable in $\Sigma$ and a type $345$ list-assignment $L$ for $G$, finds an $L$-colouring of $G$ (if it
exists) in $O(\log v(G))$ rounds.
\end{thm}

Our main algorithmic result is as follows.

\begin{thm}
\label{algorithmicresults}
For each $t \in \{4,5\}$, there exists a deterministic LOCAL distributed algorithm for $t$-list-colouring a $K_{t}$-minor-free graph in $O(\log v(G))$ rounds. For $t$ large enough, there exists a deterministic LOCAL distributed algorithm for $13 \cdot \max\left\{h(t),\left\lceil \frac{31}{2}(t-1) \right\rceil \right\}$-list-colouring a $K_{t}$-minor-free graph on $n$ vertices in $O(\log v(G))$ rounds. 
\end{thm}

These theorems build on the ideas of the second author from~\cite{Postlealgorithms}. In fact the algorithm is identical to the algorithm given in~\cite{Postlealgorithms}. The algorithm relies on the existence of what the second author called \emph{wallets} in his work on graphs embeddable on a surface. Therefore the main algorithmic work in this paper is showing that under the conditions of Theorem~\ref{algorithmicresults}, suitable wallets exist; this in turn relies on new deletability results for $K_t$-minor-free graphs.

For convenience, we also define the notions of boundary and coboundary which we will use throughout.
\begin{definition}
     Let $G$ be a graph, and $H$ an induced subgraph of $G$.  The \emph{boundary} of $H$ in $G$ is $\{v \in V(H): N(v) \setminus V(H) \neq \emptyset$\}. The \emph{coboundary} of $H$ in $G$ is the boundary of $G[V(G) - V(H)]$. For $C>0$, we say $H$ is a \textit{$C$-pocket} if $H$ is connected, $v(H) \leq C$, and $d_{G}(v) \leq C$ for every $v \in V(H)$. For $k >0$, we say $H$ is \textit{$k$-deep} if the coboundary of $H$ is non-empty and has size at most $\frac{v(H)}{k}$.
\end{definition}

The paper is structured as follows. In Section~\ref{smalltsection}, we prove the deletability results for $K_{t}$-minor-free graphs when $t \leq 5$. In Section~\ref{largevalueoftprecolouring}, we prove analogous deletability results for large values of $t$. In Section~\ref{exponentialgrowthandlocalhadwigers}, we prove Theorem~\ref{smallvaluesoft} and Theorem~\ref{largevaluesoft}. In Section~\ref{clutchExist}, we prove Theorem~\ref{algorithmicresults}. Finally, in Section~\ref{lowerboundssection}, we show that $(t-1)$-colours does not suffice for Local Hadwiger's Conjecture. 

\section{Deletable subgraphs in graphs with no $K_{t}$ minor for small values of $t$}
\label{smalltsection}
In this section we find deletable subgraphs in graphs with no $K_{t}$ minor for $t \leq 5$. 
We start off by investigating $K_{3}$-minor-free graphs, which are simply forests. The proof is easy, but the proof can be viewed as a template for the $K_{4}$-minor-free and $K_{5}$-minor-free results. 

\begin{thm}
\label{K3deletability}
Let $G$ be a forest and $X \subseteq V(G)$. If $G$ contains no $3$-deletable subgraph disjoint from $X$, then $v(G) \leq 2(|X|-1)$.

\end{thm}

\begin{proof}
Suppose not, and let $G$ be a vertex minimal counterexample. First we claim that $G$ is a tree. If not, let $C_{1},\ldots,C_j$ be components of $G$, $j \geq 2$, and let $X_{i} = C_{i} \cap X$ for $i \in \{1,\ldots,j\}$. As $G$ contains no $3$-deletable subgraph disjoint from $X$, $C_{i}$ contains no $3$-deletable subgraph disjoint from $X_{i}$ for all $i \in \{1,\ldots,j\}$. Therefore by minimality, we have that $v(C_{i}) \leq 2(|X_{i}| -1)$ for all $i \in \{1,\ldots,j\}$. Thus $v(G) = \sum_{i=1}^{j} v(C_{i}) \leq \sum_{i=1}^{j} 2(|X_{i}|-1) \leq 2(|X|-1)$, as desired. 

Now we claim that all vertices in $X$ are leaves. Suppose not, and let $v \in X$ such that $d_{G}(v) \geq 2$. As $v$ is not a leaf, $G-v$ contains components $C_{1},\ldots,C_{j}$ for $j \geq 2$. Let $X_{i} = X \cap V(C_{i})$ for $i \in \{1,\ldots,j\}$. Note that $C_{i}$ cannot contain a $3$-deletable subgraph disjoint from $X_{i}$, as otherwise $G$ has a $3$-deletable subgraph disjoint from $X$. Thus $v(C_{i}) \leq 2(|X_{i}| -1)$. Therefore $v(G) = 1 + \sum_{i =1}^{j} v(C_{i}) \leq 1 + \sum_{i=1}^{j}2(|X_{i}|-1) \leq 2(|X|-1)$. 

Now we argue that $G$ contains no vertex of degree at most $2$ which is not in $X$. Suppose not, and let $v \in V(G)$ such that $d_{G}(v) \leq 2$ and $v \not \in X$. In this case, $v$ is a $3$-deletable subgraph of $G$ and is disjoint from $X$, a contradiction. 

Thus $G$ is a tree which contains no vertices of degree $2$, and all vertices of degree $1$ are in $X$. By the handshaking lemma it follows that:
\[\sum_{v \in V(G)} d_{G}(v) = 2v(G)-2.\]
So,
\[\sum_{v \in V(G), d_{G}(v) \geq 3} d_{G}(v) +|X| = 2v(G)-2.\]
Therefore, we have:
\[3(v(G)-|X|) + |X| \leq  2v(G)-2,\]
which implies
\[v(G) \leq 2|X|-2.\]
as desired. 
\end{proof}

We extend the ideas of the above proof to graphs with no $K_{4}$ minor. While $K_{4}$-minor-free graphs are not as simple as trees, they are known to be tree-like in the sense of treewidth. We recall the needed notions:

\begin{definition}
A \textit{tree decomposition} of a graph $G$ is a pair $(T,\beta)$ where $T$ is a tree with vertices $x_{1},\ldots,x_{n}$,  each vertex $x_{i}$ has an associated bag $B(x_{i}) \subseteq V(G)$ and $\beta$ is the set of bags satisfying the following properties:
\begin{enumerate}
\item{For every $v \in V(G)$, there exists a bag $B$ such that $v \in B$.}
\item{If the set of all bags containing $v \in V(G)$ is  $B_{1},\ldots,B_{j}$, then the subgraph of $T$ induced by the vertices associated to $B_{1},\ldots,B_{j}$ is connected.}
\item{For every edge $uv \in E(G)$, there exists a bag which contains both $u$ and $v$.}
\end{enumerate}
\end{definition}

\begin{definition}
The \textit{treewidth} of $G$ is minimum over all tree decompositions $(T,\beta)$ of $G$, $\max_{v\in V(T)}(|\beta(v)|-1)$.
\end{definition}

It is well known that $K_{4}$-minor-free graphs are exactly the graphs with treewidth at most $2$. 

\begin{thm}[\cite{diestel}]
A graph $G$ is $K_{4}$-minor-free if and only if $G$ has treewidth at most $2$. 
\end{thm}

It will be useful to work with so called smooth tree decompositions.

\begin{definition}
A tree decomposition is \textit{$k$-smooth} if each bag has exactly $k+1$ vertices and two adjacent bags share exactly $k$ vertices. 
\end{definition}

\begin{thm}[\cite{smoothtreewidth}]
If a graph $G$ has a tree decomposition of width $k$, then $G$ has a $k$-smooth tree decomposition.
\end{thm} 

Finally, recall that a $k$-separation $(G_{1},G_{2})$ is a partition of $V(G)$ such that $V(G_{1}) \cap V(G_{2})$ has at most $k$ vertices, $V(G_{i}) \setminus V(G_{j}) \neq \emptyset$, for $i \neq j \in \{1,2\}$ and $V(G_{1}) \cup V(G_{2}) = V(G)$. 

Now we are ready to prove the existence of deletable subgraphs in $K_{4}$-minor-free graphs.

\begin{thm}
\label{K4minorfree}
Let $G$ be a $K_{4}$-minor-free graph and let $X \subseteq V(G)$. If $G$ contains no 4-deletable subgraph disjoint from $X$, then $v(G) \leq 11|X|$.
\end{thm}

\begin{proof}
Suppose not, and let $G$ be a vertex minimal counterexample.

\begin{claim}
\label{3deg}
All vertices of degree at most $3$ are in $X$.
\end{claim}
\begin{claimproof}
 Suppose not. If $v \in V(G)$ such that $d_{G}(v) \leq 3$, then $v$ is $4$-deletable since $K_1$ is $1$-colourable. If $v \not \in X$, this contradicts our assumption.
\end{claimproof}

Now fix a tree decomposition $(T,\beta)$ of $G$ which is $2$-smooth (which exists as $G$ is $K_{4}$-minor-free).  

\begin{claim}
\label{leafX}
If $B \in \beta$ is a leaf bag in the tree decomposition, then $B$ contains at least one vertex in $X$. Furthermore, this vertex is not contained in any other bag.
\end{claim}

\begin{claimproof}
Since $B$ is a leaf, it is adjacent to exactly one other bag $B'$. By 2-smoothness, there is exactly one vertex in $B\setminus B'$. Call it $v$. Since $v \not \in B'$, $v$ is not in any bag except $B$. Hence all the neighbours of $v$ must lie in $B$. Thus, $v$ has degree at most $2$. By Claim \ref{3deg}, this means that $v\in X$. 
\end{claimproof}

\begin{claim}
\label{pathX}
If $T$ has a path of 6 nodes of degree two, then the union of the corresponding bags contains at least one vertex in $X$. This vertex is not contained in any bag other than these six bags.
\end{claim}

\begin{claimproof}
	Suppose there is a path $t_1,t_2,t_3,t_4,t_5,t_6$ of adjacent nodes of degree 2 such that this does not hold. Let $B_i$, for $i\in \{1,\ldots,6\}$ be the corresponding bags. Let $B_0$ be the bag corresponding to the other neighbour of $t_1$ and $B_7$ the bag corresponding to the other neighbour of $t_6$. Let $B_0\cap B_1=\{u,v\}$ and let $w$ be the third vertex in $B_1$. If $w\notin B_1\cap B_2$, then $N(w)\subseteq B_1$. Hence $w$ has degree at most 2 and is in $X$ by Claim \ref{3deg}. Thus $w\in B_2$.
	
	Without loss of generality, suppose $B_1\cap B_2=\{v,w\}$. Let $x$ be the third vertex in $B_2$. By a similar argument, if $x\notin B_3$, then $x$ has degree at most $2$, and thus $x\in X$. Thus we have two cases. Either $B_2\cap B_3=\{v,x\}$ or $B_2\cap B_3=\{w,x\}$. In the first case, all the neighbours of $w$ must lie in $B_1\cup B_2$. Thus $w$ has degree at most 3 and is thus in $X$. Thus $B_2\cap B_3=\{w,x\}$.
	
	We now repeat the argument in the previous paragraph for $B_3$ through $B_6$, giving us the following for some distinct vertices $t,u,v,w,x,y,z,a,b,c\in V(G)$:
	\[\begin{array}{l l}
		B_0=\{t,u,v\} & B_1=\{u,v,w\}\\
		B_2=\{v,w,x\} & B_3=\{w,x,y\}\\
		B_4=\{x,y,z\} & B_5=\{y,z,a\}\\
		B_6=\{z,a,b\} & B_7=\{a,b,c\}
	\end{array}\]
	We now claim that $G'=G[w,x,y,z]$ is a 4-deletable subgraph of $G$. Note that $G'$ is isomorphic to $K_{4}-e$, otherwise we have a vertex of degree at most $3$, which then lies in $X$. Further,  observe that $w$ and $z$ are adjacent in $G$ to at most two vertices in $G-G'$. Note $x$ and $y$ are adjacent in $G$ to at most one vertex in $G-G'$. Therefore to prove that $G'$ is $4$-deletable, it suffices to see that $K_{4}-e$ is $L$-colourable, when the two vertices of degree $2$ have lists of size at least two, and the two vertices of degree $3$ have lists of size at least $3$.  It is easily checked that this graph is $L$-colourable, and thus $G'$ is $4$-deletable. Hence $G'$ contains a vertex of $X$, as desired. 
\end{claimproof}

\begin{claim}
	\label{treesub}
	Suppose $T''$ is a tree with $n$ vertices, $\ell$ leaves, and let $p$ denote the maximum number of disjoint paths of length $6$, and containing only vertices of degree $2$ in $T'$. Then $n\leq 11(\ell+p)-3$.
\end{claim}
	
\begin{claimproof}
	There exists a tree $T'$ with no vertices of degree 2 such that $T''$ is a subdivision of $T'$. The set of leaves is the same in $T'$ and $T''$. The same is true of the set of vertices of degree more than 2. Note that there are at most $\ell$ vertices of degree more than 2 in $T''$. Additionally, vertices of degree 2 in $T''$ only appear from subdividing edges of $T'$. For every edge $e$ in $T'$, let $s_e$ be the number of vertices of degree $2$ in $T''$ which lie on the unique path between the endpoints of $e$ in $T''$. Thus every edge $e$ in $T'$ contributes $\lfloor \frac{s_e}{6}\rfloor\geq \frac{s_e}{6}-1$ to $p$ and $s_e$ to $n$. Thus $n\leq 2\ell+\sum_{e\in E(T')}s_e\leq 2\ell +6(p+1)\leq 6(\ell+p)+2\leq 11(\ell+p)-3$ as required. The constant term comes from the fact that $\ell \geq 1$.
\end{claimproof}

\vskip.1in

Now we finish the proof.  Defining $p$ and $\ell$ for $T$ as in Claim \ref{treesub}, Claims \ref{leafX} and \ref{pathX} tell us that $|X|\geq p+\ell$. Furthermore, if we let $n$ be the number of nodes in $T$, $v(G)=n+2$. Thus, $v(G)=n+2\leq 11(\ell+p)\leq 11|X|$ as required.
\end{proof}

Unfortunately for $t \geq 5$, the class of $K_{t}$-minor-free graphs has unbounded treewidth, and so the above arguments no longer suffice. Nevertheless, $K_{5}$-minor-free graphs still look mostly treelike due to Wagner's Theorem, which will suffice for our purposes. We need some definitions before we can state Wagner's Theorem.

\begin{definition}
The Wagner graph, denoted $V_{8}$, has vertices $v_{1},v_{2},v_{3},\ldots,v_{8}$, edges $\{v_{i}v_{i+1}~|~i \in \{1,\ldots,8\}\}$ (indices taken modulo $8$) and $\{v_{i}v_{i+4}~|~i \in \{1,\ldots 4\}\}$.
\end{definition}

\begin{definition}
Let $G_{1},G_{2}$ be graphs with cliques $K_{1}$ and $K_{2}$ respectively, such that $v(K_{1}) = v(K_{2}) =k$. A $k$-sum of $G_{1}$ and $G_{2}$ is the graph obtained by identifying the vertices in $K_{1}$ with the vertices in $K_{2}$, and then possibly deleting edges if desired. 
\end{definition}

With this we can state Wagner's Theorem. 

\begin{thm}[\cite{Wagnersthm}]
Every $K_5$-minor-free graph can be obtained from a collection of planar graphs and $V_8$ using clique sums over cliques of size at most 3.
\end{thm}

We will need the following theorem of the second author:

\begin{thm}[\cite{Hyperbolicfamilies}]
\label{Lukestheorem}
	There is a positive integer $c$ such that if $G$ is planar and $X\subseteq V(G)$ such that $G$  has no 5-deletable subgraph disjoint from $X$, then $v(G)\leq c|X|$.
\end{thm}

We will also need the following list-colouring result of \v{S}krekovski.

\begin{thm}[\cite{Skr}]
\label{K5choosability}
Let $G$ be a $K_{5}$-minor-free graph and $H$ a subgraph of $G$ isomorphic to either $K_{1},K_{2},$ or $K_{3}$. For any list assignment $L$ such that for all $v \in V(G) \setminus V(H)$ we have $|L(v)| \geq 5$ and otherwise $|L(v)| \geq 1$. If for all $x,y \in V(H)$, we have $L(x) \neq L(y)$, then there exists an $L$-colouring of $G$. 
\end{thm}

Note Wagner's Theorem implies that if $G$ is $K_{5}$-minor free, then there exists a spanning supergraph $G^{*}$ which admits a tree decomposition $(T,\beta)$ with the following properties. Each bag either induces a planar graph or $V_{8}$,  if $b_{1}b_{2} \in E(T)$, and $B_{1},B_{2}$ are the associated bags, then $|V(B_{1}) \cap V(B_{2})| \leq 3$, and the shared vertices induce a clique in $G^{*}$ in $B_{1}$ and respectively $B_{2}$. Call such a decomposition a \textit{Wagner tree decomposition}. We will want to massage this tree decomposition into something more controllable.

\begin{lemma}
\label{Nicedecomposition}
Every $K_{5}$-minor-free graph admits a Wagner tree decomposition $(T,\beta)$ such that the following properties hold
\begin{itemize}
\item{If $b_{1}b_{2} \in E(T)$, and the associated bags $B_{1},B_{2}$ both induce planar graphs then $|V(B_{1}) \cap V(B_{2})|=3$}
\item{If $b_{1}$ has degree $2$, with neighbours $b_{2},b_{3}$, and the bag associated to $b_{1}$ is $B_{1}$ and is planar, then the bags associated to $b_{2}$ and $b_{3}$ induce graphs isomorphic to $V_{8}$.}
\end{itemize}
\end{lemma}

\begin{proof}
Suppose not. Let $G$ be $K_{5}$-minor free. Let $(T',\beta')$ be a Wagner tree decomposition of $G$.  First, suppose that $b_{1}b_{2} \in E(T')$ and $B_{1}$, $B_{2}$ are the bags associated with $b_{1}$ and $b_{2}$, where both $B_{1}$ and $B_{2}$ induce planar graphs. If $|V(B_{1}) \cap V(B_{2})| \leq 2$, then create a new tree decomposition $T$ where $V(T) = (V(T') \setminus \{b_{1}, b_{2}\}) \cup b_{1,2}$ where $b_{1,2}$ is a new node, $N(b_{1,2}) = (N(b_{1}) \cup N(b_{2}))\setminus\{b_{1},b_{2}\}$, and the vertices of the bag for $b_{1,2}$, say $B_{1,2}$ is the union of $B_{1}$ and $B_{2}$. Note $B_{1,2}$ induces a planar graph. Let $\beta= \beta' \setminus\{B_{1},B_{2}\}\cup \{B_{1,2}\}$. Then $(T,\beta)$ is a Wagner tree decomposition, as desired. Repeating this procedure implies the first item.

For the second property, suppose $b_{1}$ has degree $2$, the bag associated to $b_{1}$, $B_{1}$ induces a planar graph, and neighbours $b_{2}$ and $b_{3}$, where the bag associated to $b_{2},$ $B_{2}$ induces a planar graph. Then $|V(B_{1}) \cap V(B_{2})| =3$. Let $K = V(B_{1}) \cap V(B_{2})$. If the clique sum of $B_{1}$ and $B_{2}$ is planar, then we do as above and get a new Wagner tree decomposition. 

If not, then without loss of generality $K$ induces a separating triangle $K'$ in $B_{1}$. Let $G_{1}$ and $G_{2}$ be the interior and exterior $K'$ with respect to a planar embedding of $G[B_{1}]$, including $K'$ in both cases. Without loss of generality, we may assume that $V(B_{1}) \cap V(B_{3}) \subseteq V(G_{1})$ (as otherwise we simply relabel the graphs). Then let $(T,\beta)$ be the tree decomposition where $V(T) = V(T') \setminus \{b_{1}\} \cup \{g_{1},g_{2}\}$ where $g_{1},g_{2}$ are new vertices where $g_{1}$ is adjacent to all of the vertices $b_{1}$ is, and $g_{2}$ is adjacent to $g_{1}$, the bag associated to $g_{1}$ is $G_{1}$ and the bag associated to $g_{2}$ is $G_{2}$. Then $T$ is a Wagner tree decomposition, and now $g_{1}$ has degree $3$. Repeating this procedure gives the second item. 
\end{proof}

\begin{thm}\label{K5minorfree}
There is a constant $c'$ such that if $G$ is $K_{5}$-minor-free and $X \subseteq V(G)$, such that any $5$-deletable subgraph of $G$ contains a vertex in $X$, then $v(G)\leq c'\abs{X}$.
\end{thm}

\begin{proof}
Let $c' = 61c$ where $c$ is the constant from Theorem \ref{Lukestheorem}. Suppose the theorem is false, and let $G$ be a vertex-minimal counterexample. Let $(T,\beta)$ be a Wagner tree decomposition with the properties from Lemma \ref{Nicedecomposition}. Let $b_{1},\ldots,b_{v(T)}$ be the vertices of $T$, associated with bags $B_{1},\ldots,B_{v(T)}$ respectively. Let $\delta(B_{i})$ denote the boundary of $B_{i}$, that is the set of vertices of $B_{i}$ which lie in a bag $B_{j}$ such that $b_{i}b_{j} \in E(T)$. Let $X_{i}=X\cap (B_{i}\setminus\delta(B_i))$.

\begin{claim}
\label{deg4claim}
Every vertex $v$ with $d_{G}(v) \leq 4$ is in $X$.
\end{claim}
\begin{claimproof}
Such a vertex is $5$-deletable and thus must lie in $X$.
\end{claimproof}
\begin{claim}
If $b_{1}$ is a leaf of $T$, then $X_{1} \neq \emptyset$. 
\end{claim}
\begin{claimproof}
Suppose not. If $B_{1}$ induces a graph isomorphic to $V_{8}$, then the result follows since $\delta(B_{1})$ contains at most three vertices, the fact that $V_{8}$ is $3$-regular, and Claim \ref{deg4claim}. Thus we can assume that $B_{1}$ induces a planar graph. Let $B_1'$ be obtained by from the graph induced by $B_{1}$ by making $\delta(B_1)$ a clique. Note $B_{1}'$ is planar by definition of Wagner tree decompositions. We claim that $ B'' := B_{1}' - \delta(B_{1}')$ is $5$-deletable. To see this, let $L$ be any list assignment such that each vertex $v$ of $B''$ receives a list with at least $5 - (d_{G}(v) - d_{G[B'']}(v))$ colours. Note that each vertex has a list of size at least $2$, since the boundary has at most three vertices. Theorem \ref{K5choosability} implies that $B''$ is $L$-choosable, since we can extend $L$ to a $5$-list assignment, and then precolour the vertices of $\delta(B_{1}')$ such that for the resulting list assignment $L'$, $B''$ is $L$-colourable if and only if $B_{1}'$ is $L'$-colourable. 
Thus $B''$ is $5$-deletable, and thus $X_{1}$ is non-empty, as desired. 
\end{claimproof}

\vskip.1in
We now split into cases. First suppose that $T$ has no vertices of degree $2$.
As $T$ has no vertices of degree $2$, if $L$ is the number of leaves of $T$ we have that $v(T) \leq 2L$ (by the same argument given in the $K_{3}$-minor-free case). Thus $v(T) \leq 2|X|$ as each leaf bag contains at least one unique vertex of $X$. As $G$ contains no $5$-deletable subgraphs disjoint from $X$, it follows that $G[B_{i}]$ contains no $5$-deletable subgraph disjoint from $X \cup \delta(B_{i})$. To see why, suppose $G[B_{i}]$ contained a $5$-deletable subgraph $H$ disjoint from $X \cup \delta(B_{i})$. As the degrees of vertices not in the boundary of $\delta(B_{i})$ are the same in $G[B_{i}]$ as in $G$, it follows that $H$ is $5$-deletable in $G$, and disjoint from $X$, a contradiction. If $B_i$ induces a planar graph, then $|B_i|\leq c(\abs{X_i\cup \delta(B_i)})$ by Theorem \ref{Lukestheorem}. If $B_i=V_8$, then every vertex in $B_i$ must be in $\delta(B_i)$ or in $X_i$ by Claim \ref{deg4claim} since $V_8$ is $3$-regular. Thus, in this case, we also find that $|B_i|=\abs{X_i\cup \delta(B_i)} \leq c(\abs{X_i\cup \delta(B_i)})$. Now we compute:

\begin{align*}
v(G) &\leq \sum_{i=1}^{v(T)} |B_{i}| \\
& \leq  \sum_{i=1}^{v(T)} c(|X_{i}| + |\delta(B_{i})|)  \\
&\leq c|X| + c \cdot \sum_{i=1}^{v(T)}|\delta(B_{i})| \\
& \leq c|X| + c(3(2v(T)-2))\\
& \leq c|X| + c(12|X| - 6) \\
& \leq c'|X|,
\end{align*}
as desired, where for the last inequality we used that $c' \ge 13c$. 

Therefore we can assume that $T$ has vertices of degree $2$. Observe by our choice of tree decomposition that if $b_{1},b_{2},b_{3},\ldots,b_{r}$ is a path in $T$ where all vertices have degree $2$, then if $G[B_{i}]$ is planar, both $B_{i+1}$ and $B_{i-1}$ are isomorphic to $V_{8}$. Further, if $G[B_{i}]$ is isomorphic to $V_{8}$ and $b_{i}$ has degree $2$, then as $V_{8}$ has no triangles, $|\delta(B_{i})| \leq 4$, and hence $|X_{i}| \geq 4$. Let $T'$ denote the tree obtained from $T$ by contracting all paths consisting only of vertices of degree 2. For every $e\in T'$, let $s_e$ be the number of vertices on the unique path in $T$ between the endpoints of $e$.  Let $D_{2}$ be the number of degree $2$ vertices in $T$, and $D_{1}$ be the number of leaves. each edge in $T'$ contributes $s_e$ to $D_2$ and at least $4\lfloor\frac{s_e}{2}\rfloor\geq 2s_e-2$ to $X$. There are at most $2D_1-1$ edges in $T'$ so the contribution of the degree 2 vertices to the size of $X$ is at least $2D_2-4D_1+2$. So $\abs{X}\geq 2D_2-4D_1+D_1=2D_2-3D_1$ and $\abs{X}\geq D_1$. So $10\abs{X}\geq 4D_2+2D_1$. By the handshaking lemma we have
\begin{align*}
\sum_{v \in V(T)} d_{T}(v) = 2v(T) -2.
\end{align*}

Rewriting this we have: 

\begin{align*}
\sum_{v \in V(T), d_{T}(v) \geq 3} d_{T}(v) + 2D_{2} + D_{1} = 2v(T)-2,
\end{align*}

and
\begin{align*}
\sum_{v \in V(T), d_{T}(v) \geq 3} d_{T}(v) + 2D_{2} + D_{1} \geq 3(v(T)-2D_{2}-D_{1}) +2D_{2} + D_{1}.
\end{align*}

Rearranging the equations we get
\[v(T) \leq 4D_{2} + 2D_{1} -2 \leq 10|X|\]

Now we do a similar calculation as above:

\begin{align*}
v(G) &\leq  \sum_{i=1}^{v(T)} |B_{i}| \\
& \leq  \sum_{i=1}^{v(T)} c(|X_{i}| + |\delta(B_{i})|)\\
& \leq  c|X| + c\cdot \sum_{i=1}^{v(T)}|\delta(B_{i})| \\
& \leq c|X| + c(3(2v(T)-2))\\
& \leq c|X| + c(60|X| - 6) \\
& \leq c'|X|,
\end{align*}
as desired, where for the last inequality we used that $c' \ge 61c$. 
\end{proof}

\section{Deletable subgraphs in $K_{t}$-minor-free graphs for large $t$.}
\label{largevalueoftprecolouring}

In this section we prove similar theorems to last section, for large values of $t$ (at the cost of worse bounds). 

We need the following key lemma of the second author and Norin~\cite{norinpostlelist} (recall that for a set of edges $E$, $G / E$ denotes the graph obtained from $G$ by simultaneously contracting all edges in $E$):

\begin{lemma}[\cite{norinpostlelist}]
\label{matchinglemma}
If $G$ is a graph with minimum degree $d \geq 6k$, then there exists a non-empty $X \subseteq V(G)$ and a matching $M$ from the coboundary $Y$ of $X$ to $X$ that saturates $Y$ such that $|Y| \leq 3k$ and $G[X \cup Y] / M$ is $k$-connected. 
\end{lemma}

We also need the following theorem of B\"{o}hme, Kawarabayashi, Maharry and Mohar:

\begin{thm}[\cite{Bohme2009}]
\label{smallnumberofvertices}
For every $t\geq 0$, there is an integer $N(t)$ such that if $G$ is a $\left\lceil \frac{31}{2}(t+1)]\right\rceil$-connected graph and $v(G)>N(t)$, then $G$ has a $K_t$ minor.
\end{thm}

To allow what follows to be more robust in the light of ongoing work towards Hadwiger's conjecture, we will define $h$ to be a function such that if $G$ has no $K_t$ minor, it is $h(t)$-list-colourable, where $h(t)$ is the smallest possible value.

Unlike in the previous sections, we prove a slightly stronger statement to allow the proof to go through. This requires a definition.
For a graph $G$, a set $X \subseteq V(G)$, and $C$ a constant, we let $m^{G}_{C}(X) = \sum_{v \in X} \min\{d_{G}(v),C\}$.

We prove the following: 

\begin{thm}\label{K_tDeletability}
Let $G$ be a $K_{t}$-minor free graph, $N(t)$ be the integer from Theorem \ref{smallnumberofvertices}, $L$ be a $13k$-list-assignment of $G$ where $k=\max \{h(t),\lceil \frac{31}{2}(t+1)\rceil \}$, and $X \subseteq V(G)$ be non-empty. If $G$ has no $13k$-deletable subgraph disjoint from $X$, then $v(G) \leq  N(t)m^{G}_{6k}(X)$. 
\end{thm}

\begin{proof}
Suppose not. Let $G$ and $X \subseteq V(G)$ be a counterexample such that $v(G)$ is minimized, subject to this, $|X|$ is maximized, and subject to that, $e(G[X])$ is minimized. We seek to prove that $v(G)\leq N(t)m^{G}_{6k}(X)$. 

\begin{claim}
\label{edgeless}
The graph $G[X]$ is edgeless.    
\end{claim}
\begin{claimproof}
Suppose not, and let $e \in E(G[X])$. If $G-e$ contains a $13k$-deletable subgraph disjoint from $X$, say $P$, then $P$ does not contain either endpoint of $e$, as $e \in E(G[X])$. 
But then it follows immediately from the definition that $P$ is a $13k$-deletable subgraph in $G$ and disjoint from $X$, a contradiction. 
\end{claimproof}

Now consider $Y \subseteq X$ such that all vertices in $Y$ have degree in $G$ strictly less than $6k$ in $G$.

\begin{claim}
\label{degreeclaim}
If $v \in V(G) - Y$, then $v$ has at most $6k$ neighbours in $Y$. Further, $G-Y$ has minimum degree at least $6k$. 
\end{claim}
\begin{claimproof}
If $v \in X \setminus Y$, then as $G[X]$ is edgeless by Claim \ref{edgeless}, by definition $v$ has degree at least $6k$ in $G - Y$. Therefore consider a vertex $v \in V(G) \setminus X$. Observe that in $G$, $v$ has degree at least $13k$. If not, then $v$ is a $13k$-deletable subgraph disjoint from $X$, a contradiction. We claim that $v$ has at most $6k$ neighbours in $Y$, and hence degree at least $7k$ in $G \setminus G[Y]$. Suppose towards a contradiction that $v$ has at least $6k+1$ neighbours in $Y$. Consider the vertex set $X' = X \cup \{v\}$, and the graph $G'$ which is obtained by deleting all edges between $v$ and $X$. Suppose that $G'$ has a $13k$-deletable subgraph disjoint from $X'$, say $P$. It then follows that $P$ is a $13k$-deletable subgraph of $G$, and since $X \subseteq X'$, $P$ is disjoint from $X$. Therefore there are no $13k$-deletable subgraphs of $G'$ disjoint from $X'$. Thus since $X$ was picked maximally, we have that $v(G') \leq N(t)m^{G'}_{6k}(X')$.  Now observe that $m^{G'}_{6k}(X') \leq m^{G}_{6k}(X) -1$, since for every vertex $y \in Y$, we have $\min \{d_{G}(v),6k\} = d_{G}(v)$, and as $v$ has at least $6k+1$ edges to vertices in $Y$, which are deleted in $G'$, causing each of these vertices to have their degree drop by one. This creates a contradiction as now we have $v(G) = v(G') \leq N(t)m^{G'}_{6k}(X') \leq N(t)m^{G}_{6k}(X)$, as desired.
\end{claimproof}

\vskip.1in

Therefore $G-Y$ has minimum degree at least $6k$. Now by applying Lemma \ref{matchinglemma} to $G$ there exists a set $A \subseteq V(G) - Y$ and a matching $M$ from the coboundary $B$ of $A$ to $A$ that saturates $B$ such that $|B| \leq 3k$ and $(G-Y)[A \cup B] /M$ is $k$-connected.

Note Theorem \ref{smallnumberofvertices} implies that $|A|+|B|-|M|\leq N(t)$. For if not, $(G-Y)[A\cup B]/M$ contains a $K_{t}$ minor, and thus $G$ contains a $K_{t}$ minor, a contradiction. Further, as $M$ saturates $B$, it follows that $|M|=|B|$. Thus $|A|\leq N(t)$. Now we break into cases depending on how many vertices of $X$ are in $A$.

\vskip.1in
\noindent \textbf{Case 1: $A$ contains at least $3k+1$ vertices of $X$}
\vskip.1in

In this case we claim we can find a $13k$-deletable subgraph in $G$ disjoint from $X$. Let $G'=G-G[A]$ and let $X'=(X\cup B)\setminus A$. Note that $\abs{X'}<\abs{X}$ as $A$ contains at least $3k+1$ vertices of $X$, and $|B| \leq 3k$. We claim $G'$ has no $13k$-deletable subgraph disjoint from $X'$. Suppose for a contradiction $G'$ had such an induced subgraph $P$. Immediately we have $P$ is disjoint from $X$, since $P$ is a subgraph of $G'$, and $X'$ contains all vertices of $X$ except those in $A$.  In $G$, as $P$ is disjoint from $B$ and $Y$, no vertex in $P$ is adjacent to a vertex in $A$, as $B$ is the coboundary of $A$. It follows that $P$ is $13k$-deletable and disjoint from $X$ in $G$, a contradiction. Thus by minimality of $G$, we have that $v(G') \leq N(t)m^{G'}_{6k}(X')$. Thus
\begin{align*}
v(G) &\leq N(t)m^{G'}_{6k}(X') +|A| \\
&\leq N(t)\left(\sum_{v \in B} \min\{d_{G'}(v),6k\} + \sum_{v \in X' -B} \min\{d_{G'}(v),6k\}\right) + |A| \\
&\leq N(t)\left(6k|B| + \sum_{v \in X'-B} \min\{d_{G'}(v),6k\}\right) + |A| \\
&\leq N(t)\left(6k|A \cap X| + \sum_{v \in X'-B} \min\{d_{G'}(v),6k\}\right) +|A| - 6kN(t) \\
&\leq N(t)m^{G}_{6k}(X).
\end{align*}

These inequalities follow since $|A \cap X| \geq 3k+1$, $|B| \leq 3k$,  $|A| \leq N(t)$, and all of the vertices in $A \cap X$ are not in $Y$. But now we have a contradiction, so $|A| \leq 3k$. 

\vskip.1in
\noindent \textbf{Case 2: $\abs{A \cap X} \leq 3k$ and $A \setminus X \neq \emptyset$}
\vskip.1in

Let $P=G[A-X]$. We claim that $P$ is a $13k$-deletable subgraph of $G$ disjoint from $X$. Indeed, every vertex in $P$ has at most $6k$ neighbours in $Y$ by Claim \ref{degreeclaim}, at most $3k$ neighbours in $A\cap X$ and at most $3k$ neighbours in $B$. It follows that each vertex in $P$ has at most $12k$ neighbours outside of $P$. Therefore to show that $P$ is $13k$-deletable, it suffices to show that given lists of size at least $k$, $P$ is $k$-list-colourable. But $k \geq h(t)$, so $P$ is indeed $k$-list-colourable, a contradiction.

\vskip.1in
\noindent \textbf{Case 3: $A \subseteq X$ and $\abs{A \cap X} \leq 3k$}
\vskip.1in

For any vertex $v \in A$, as $G[X]$ is edgeless, it follows that in $G-Y$ the vertex $v$ has degree at most $3k$. Hence as $G[X]$ is edgeless, we again get that $v$ has degree at most $3k$ in $G$, but by construction $v \not \in Y$, so $v$ has degree at least $6k$ in $G$, a contradiction. 
 
As this covers all possibilities, we have derived a contradiction and hence the theorem follows. 
\end{proof}

\section{List-colouring Locally-$K_{t}$-minor-free graphs} 
\label{exponentialgrowthandlocalhadwigers}
In this section we prove Theorem \ref{smallvaluesoft} and Theorem \ref{largevaluesoft}. We prove both Theorem \ref{smallvaluesoft} and Theorem \ref{largevaluesoft} together. As it will be relevant for the proof, let $N^{i}(v)$ denote the vertices at distance exactly $i$ from $v$. 
\begin{thm}
For each $t \geq 3$, there exists a constant $c_{t}$ such if a graph $G$ is  $\lceil c_{t}\log(v(G)) \rceil$-locally-$K_{t}$-minor-free, then it is $f(t)$-list-colourable, where $f(3)= 3$, $f(4)=4$,$f(5) = 5$ and for $t \geq 6$, we have $f(t) = 13\max\{h(t),\lceil \frac{31}{2}(t-1) \rceil\}$. 
\end{thm}
\begin{proof}
As before, let $k = \max\{h(t),\lceil \frac{31}{2}(t-1) \rceil\}$. Define the constants $c'_{t}$ where $c'_{3} = 2$, $c'_{4} = 11$, $c'_{5} = 61c$, where $c$ is the constant in Theorem \ref{Lukestheorem}, and $c'_{t} = 6N(t)k$ for $t \geq 6$, where $N(t)$ is the integer from Theorem \ref{smallnumberofvertices}.  Let $c_{t} = \frac{1}{\log(\frac{c'_{t}}{c'_{t}-1})}$. 

Suppose the theorem does not hold for this choice of $c_t$. Thus there exists $t\ge 3$, a graph $G$ and a $f(t)$-list assignment $L$ of $G$ such that $G$ is not $L$-colorable. We assume without loss of generality that every proper subgraph of $G$ is $L$-colorable. It then follows that $G$ is connected.

If $G$ is $K_{t}$-minor free, then by the choice of $f(t)$, we find that $G$ is $L$-colourable, a contradiction. So we assume that $G$ is not $K_t$-minor-free.

Fix $v\in V(G)$. Since $G$ is $\lceil c_t \log(v(G))\rceil$-locally-$K_t$-minor-free, it follows that there exists vertices in $G$ at distance more than $\lceil c_{t}\log(v(G)) \rceil$ from $v$. Let $q$ be the largest integer such that $q \leq \lceil c_{t}\log(v(G)) \rceil$. For an integer $i\ge 0$, we let $N^i(v)$ denote the vertices at distance exactly $i$ from $v$ and we let $B^{i}$ denote the subgraph of $G$ induced by vertices of $G$ at distance at most $i$ from $v$.
\begin{claim}
\label{neighbourhoodsgrowbound}
For all $i \in \{1,\ldots,q\}$, $v(B^{i}) \leq c_{t}' |N^{i}(v)|$.
\end{claim}

\begin{claimproof}
If not, by Theorems \ref{K3deletability}, \ref{K4minorfree}, \ref{K5minorfree} and \ref{K_tDeletability}, then $B^{i}$ contains a $f(t)$-deletable subgraph $P$ disjoint from $N^{i}(v)$. As $P$ is disjoint from $N^{i}(v)$, it is an $f(t)$-deletable subgraph of $G$. Since $G\setminus V(P)$ is a proper subgraph of $G$, we have that there exists an $L$-colouring of $G\setminus V(P)$. As $P$ is $f(t)$-deletable, we find that this $L$-colouring extends to an $L$-colouring of $G$, a contradiction.
\end{claimproof}

\begin{claim}\label{cl:neighbor}
For all $i \in \{1,\ldots,q\}$ we have 
\[v(B^i) \geq \frac{c_{t}'}{c_{t}'-1}\cdot v(B^{i-1}).\]
\end{claim}

\begin{claimproof}
Recall that $v(B^{0}) =1$, and $|N^{0}(v)| =1$. By Claim \ref{neighbourhoodsgrowbound}, we have 
\[v(B^{i}) \leq c_{t}' |N^{i}(v)|,.\]
Thus 
\[v(B^{i-1}) = v(B^{i}) - |N^{i}(v)| \leq (c_{t}'-1)|N^{i}(v)|.\]
Hence
\[v(B^i) = v(B^{i-1}) + |N^i(v)| \geq v(B^{i-1})+ \frac{1}{c_{t}'-1}\cdot v(B^{i-1}) = \frac{c_{t}'}{c_{t}'-1} \cdot v(B^{i-1}),\]
as desired.
\end{claimproof}

\vskip.1in

Since $v(B^0) =1$, it follows from Claim~\ref{cl:neighbor} that for all $i\in \{1,\ldots,q\}$
\[v(B^i) \geq \left(\frac{c'_{t}}{c'_{t}-1}\right)^{i} \geq 2^{\frac{i}{c_{t}}}\]
However since there exists a vertex in $G$ at distance strictly more than $q$ from $v$, we find that 
\[v(G) \ge 1+v(B^q) > v(B^q) \ge 2^{\frac{c_{t}\log(v(G))}{c_{t}}} = v(G),\]
a contradiction. 
\end{proof}

\section{Wallets and Local Algorithms}\label{clutchExist}

In this section we prove the existence of wallets in $K_{4}$ and $K_{5}$-minor-free graphs as well at wallets in $K_{t}$-minor-free graphs for large $t$. These are technical theorems, possibly of independent interest, required to show the correctness of our algorithm. We use this result to give efficient distributed algorithms for list-colouring $K_{4}$-minor-free and $K_{5}$-minor-free graphs. 

We now define a wallet.

\begin{definition}
	Let $G$ be a graph and let $C$, $c$ and $d$ be given constants. A \emph{$(C,c,d)$-wallet} is a collection $\mathcal{H}$ of disjoint induced subgraphs of $G$ with the following properties:
	\begin{itemize}
		\item Every $H\in\mathcal{H}$ is a $C$-pocket,
		\item Every $H\in \mathcal{H}$ is $c$-deletable or $d$-deep in $G$,
		\item There are no edges of $G$ with endpoints in two different elements of $\mathcal{H}$, nor any vertices in two different elements of $\mathcal{H}$ (they are non-touching) and
		\item $\abs{\mathcal{H}}\geq \frac{1}{2C}v(G)$.
	\end{itemize}
\end{definition}

This definition is slightly different than that given in \cite{Postlealgorithms}, as we do not allow purses (as our graphs are not embedded on surfaces). Nevertheless, from an algorithmic perspective, our definition is essentially the same. We will actually need a stronger notion. A \textit{deletable $(C,c)$-wallet} is a $(C,c,d)$-wallet in which each pocket is $c$-deletable. Our goal will be to find deletable $(C,c)$-wallets, which will follow easily once we find $(C,c,d)$-wallets.

We will need a lemma from \cite{Postlealgorithms} showing that minor-closed classes admit nice separators. 

\begin{lemma}[\cite{Postlealgorithms}]
\label{initialPockets}
	For every proper minor-closed family $\mathcal{F}$ and every $\epsilon>0$, there exists a constant $C$ such that the following holds: If $G$ is a graph in $\mathcal{F}$, then there exists $X\subseteq V(G)$ such that $\abs{X}\leq \epsilon v(G)$ and every component of $G-X$ is a $C$-pocket.
\end{lemma}

\subsection{$(C,4,12)$-Wallets in $K_4$-Minor-Free Graphs}

For this subsection, we assume $C$ is a constant given from Lemma \ref{initialPockets} when $\mathcal{F}$ is the family of $K_{4}$-minor-free graphs. We build towards showing that every $K_{4}$-minor-free graph contains a $(C,4,12)$-wallet.

\begin{lemma}
	Let $G$ be a $K_{4}$-minor-free graph. Every $C$-pocket with a coboundary of size 1 can be coloured with respect to any list assignment that gives 2 colours to vertices on the boundary and 4 colours to every other vertex.
\end{lemma}

\begin{proof}
	Let $H$ be a $C$-pocket which is a vertex minimal counterexample, and let $L$ be a list assignment as in the lemma statement. Let $x$ be the single vertex in the coboundary of $H$. Let $H'=G[H\cup\{x\}]$. If $H$ is acyclic, it is a tree, and as trees are $2$-list-colourable, there is a desired $L$-colouring. Thus $H$ contains a cycle $Q$. Consider the maximum number of vertex-disjoint paths from $x$ to $Q$ that also have distinct endpoints in $Q$. There are at most two of them since otherwise $H'$ contains a $K_4$ minor. Let $a$ and $b$ be the endpoints of these paths (where possibly $a =b$). Then $\{a,b\}$ form at most a $2$-vertex cut and $H'-a-b$ has at least two components. Let $A$ be the component containing $x$. By minimality, $A-x$, has an $L$-colouring.  Let $\phi$ be such a colouring. We now consider $L'=H'/(A\cup\{a,b\})$. Call the contracted vertex $y$. Then every component of $L'-y$ is a $C$-pocket of $L'$ with coboundary $\{y\}$. All of the vertices in these components had four colours in their lists. Remove $\phi(a)$ and $\phi(b)$ from the lists of vertices adjacent to $y$ in $L'$. As $L'$ is a minor of a graph that is $K_{4}$-free, $L'$ is $K_4$-minor free. Thus, by minimality of $H$, all the components of $L'-y$ can be coloured with some $L$-colouring $\psi$. Taken as an  $L$-colouring of $H'-x-a-b$, $\psi$ does not create any conflicts with $\phi$ by construction of the lists, and thus $H$ has a desired $L$-colouring. 
\end{proof}

This gives us two immediate corollaries. The first is a straightforward weakening.
\begin{cor}
	Every $C$-pocket with coboundary of size 1 is $4$-deletable.
\end{cor}

\begin{cor}
	If $P$ is a $C$-pocket with coboundary $\{a,b\}$ such that adding an edge between $a$ and $b$ in $G$ does not create a $K_4$ minor, then $C$ is $4$-deletable.
\end{cor}
This can be seen by adding the edge $ab$ and contracting it. We recall an elementary fact about the edge density of $K_{4}$-minor-free graphs. 

\begin{thm}[\cite{diestel}]
Every graph $G$ with $v(G) \geq 2$ and with no $K_{4}$ minor has $e(G) \leq 2v(G)-3$. 
\end{thm}

We are now ready to prove the main lemma of this subsection:

\begin{lemma}\label{K4ClutchExist}
	Every $K_4$-minor-free graph contains a $(C,4,12)$-wallet.
\end{lemma}

\begin{proof}
Let $G$ be a $K_{4}$-minor-free graph. We begin by applying Lemma \ref{initialPockets} with $\epsilon<\frac{1}{120}$ to obtain a set $X \subseteq V(G)$ with $|X| \leq \epsilon v(G)$ such that the components of $G-X$ are $C$-pockets. Let $\mathcal{H}$ be the set of components in $G-X$. Then $\mathcal{H}$ is a set of non-touching $C$-pockets, and we aim to extract a $(C,4,12)$-wallet from $\mathcal{H}$. Thus we focus on the elements of $\mathcal{H}$ which are not acceptable for a $(C,4,12)$-wallet. We classify the $C$-pockets of $\mathcal{H}$ which cannot belong to a $(C,4,12)$-wallet in the following way: 
\begin{itemize}
	\item Type $2$: Those with coboundary of size $2$ that are not $12$-deep and where adding an edge between the vertices in the coboundary would create a $K_4$ minor.
	\item Type $k$ for $k\geq 3$: Those with coboundary of size $k$ which are not $12$-deep.
\end{itemize}

We aim to bound the number of vertices that can lie in these pockets. We first deal with Type 2 pockets. Note $C$-pockets of Type 2 have size at most $23$ since they are not $12$-deep. We claim there is at most $23\epsilon v(G)$ vertices which lie in Type $2$ pockets. To see this, consider the graph $A$ obtained from $G$ by deleting all vertices except those in elements of $\mathcal{H}$ of Type 2 and their coboundary and then contracting each element of $\mathcal{H}$ of Type 2 to a single vertex. Then $A$ is acyclic. Otherwise, there exists a Type $2$ $C$-pocket $P$ with coboundary vertices $u,v$ such that there is a path from $u$ to $v$ in $G-P$. Thus contracting this path adds the edge $uv$, and since $G + uv$ contains a $K_{4}$ minor by definition of Type $2$, we get a contradiction. Thus there are at most $\abs{X}-1< \epsilon v(G)$ elements of $\mathcal{H}$ of Type $2$ for a maximum of $23\epsilon v(G)$ vertices in such pockets.

We now similarly bound the number of vertices in Type $k$ pockets for $k \geq 3$. Observe that there are fewer than $12k$ vertices in a $C$-pocket of Type $k$, as they are not 12-deep. As above, consider the graph $A$ obtained by deleting all vertices except those that lie in pockets of Type $k$, for all $k \geq 3$ and their coboundaries, and then contracting all pockets to a single vertex. Note that $A$ is a minor of $G$ and thus has no $K_4$ minor. Thus $e(A)\leq 2v(A)-3$. For every $k\geq 3$, let $n_k$ be the number of elements of $\mathcal{H}$ of Type $k$. Then we have $e(A)\geq \sum_{k\geq 3}kn_k\geq 3\sum_{k\geq 3} n_k$ and $v(A)\leq \abs{X}+\sum_{k\geq 3} n_k$. Thus $3\sum_{k\geq 3}n_k\leq 2\abs{X}-3+2\sum_{k\geq 3} n_k$ and so $\sum_{k\geq 3}n_k\leq 2\abs{X}-3$. Hence $v(A)\leq 3\abs{X}-3$. Applying the bound on the number of edges in a $K_{4}$-minor-free graph again, we have $\sum_{k\geq 3}kn_k\leq 6\abs{X}-9$. Putting this all together, the total number of vertices in such pockets is at most $12\sum_{k\geq 3}kn_k<108\epsilon v(G)$.

Combining the previous two paragraphs, we have that the number of vertices in elements of $\mathcal{H}$ that cannot occur in a $(C,4,12)$-wallet is at most $131\epsilon  v(G)$. Therefore let $\mathcal{H}' \subseteq \mathcal{H}$ be the subset of $\mathcal{H}$ that excludes Type $2$ and Type $k$ wallets. Observe that as $|X| \leq \epsilon   v(G)$, and the number of vertices in Type $2$ and Type $k$ wallets is at most $131\epsilon v(G)$ the number of vertices in elements of $\mathcal{H}'$ is at least $\frac{v(G)}{2}$. As each component of $\mathcal{H}$, and hence also $\mathcal{H}'$ has size at most $C$, the number of components in $\mathcal{H}'$ is at least $\frac{v(G)}{2C}$, and thus $\mathcal{H}'$ forms a $(C,4,12)$-wallet, as desired. 
\end{proof}

\subsection{$(C,5,d)$-Wallets in $K_5$-Minor-Free Graphs}
For this subsection, let $C$ be a constant given from Lemma \ref{initialPockets} when applied to the family of $K_{5}$-minor-free graphs. Throughout, $d$ will be the constant from Theorem \ref{K5minorfree}. We will prove that $K_{5}$-minor-free graphs have $(C,5,d)$-wallets, following the same outline as done for $K_{4}$-minor-free graphs.  We first state some corollaries of Theorem \ref{K5choosability} which shows that certain classes of $C$-pockets with small coboundary are $5$-deletable.

\begin{cor}
\label{easycorollaryK5}
Let $G$ be a $K_{5}$-minor-free graph with a $C$-pocket $H$. Suppose the coboundary of $H$ contains vertices $a,b,c$, where $a,b,c$ are not necessarily distinct vertices. If the graph $G'$ obtained by adding all possible edges on $\{a,b,c\}$ is $K_{5}$-minor-free, then $H$ is $5$-deletable.  
\end{cor}

\begin{proof}

Let $L$ be any list assignment of $H$ such that all vertices $v$ receive at least $5- |N(v) \cap \{a,b,c\}|$ colours. By Theorem \ref{K5choosability} $G-H$ admits an $L$-colouring $f$. 
Consider the $G' = G[V(H) \cup \{a,b,c\}]$ and extend $L$ to a list assignment $L'$ of $G'$ by giving $a$ colour $f(a)$, $b$ colour $f(b)$, $c$ colour $f(c)$, and adding $f(a),f(b),f(c)$ to any neighbour of $a,b$ or $c$ respectively. Then by Theorem \ref{K5choosability} $G'$ is $L'$-colourable, and thus $H$ is $L$-colourable, which implies that $H$ is $5$-deletable.
\end{proof}

Before proceeding to find a $(C,5,d)$-wallet, we note a useful theorem of Chen and Zhang.

\begin{thm}[\cite{bipartiteK5}]
\label{bipartiteK5bound}
If $G$ is a bipartite $K_{5}$-minor-free graph with at least four vertices, then $e(G) \leq 3v(G) -9$. 
\end{thm}

\begin{thm}
\label{K5ClutchExist}
Every $K_5$-free graph contains a $(C,5,d)$-wallet.
\end{thm}

\begin{proof}
Let $G$ be a $K_{5}$-minor-free graph. Apply Lemma \ref{initialPockets} with $\epsilon < \frac{1}{18d +19}$ to find a set $X$ such that $|X| \leq \epsilon v(G)$ and the components $\mathcal{H}$ of $G-X$ are $C$-pockets. We again will extract a $(C,5,d)$-wallet from $\mathcal{H}$. We aim to bound the number of vertices in components of $\mathcal{H}$ which cannot belong in a $(C,5,d)$-wallet.   Once again, we divide the components of $\mathcal{H}$ not suitable for a wallet into types (and our classification is complete by Corollary \ref{easycorollaryK5}). 
\begin{itemize}
	\item Type $2$: Those with coboundary of size 2 that are not $d$-deep and where adding an edge between the vertices in the coboundary would create a $K_5$ minor.
	\item Type $3$: Those with coboundary of size 3 that are not $d$-deep and where adding an edge between each pair of vertices in the coboundary would create a $K_5$ minor.
	\item Type $k$ for $k\geq 4$: Those with coboundary of size $k$ which are not $d$-deep.
\end{itemize}

For pockets of Type 2, the argument is the same as in the $K_4$-minor-free case and thus there are at most $(2d+1)\epsilon v(G)$ vertices in such pockets.

For pockets of Type $k$ for $k\geq 4$, we use a similar argument to the one used for $k\geq 3$ in the $K_4$-minor-free case. We construct the graph $A$ as follows. First contract each element of $\mathcal{H}$ of type at least $4$ down to a single vertex. Then delete all edges not incident to one of these vertices. Finally, delete all isolated vertices. Thus $A$ is a bipartite graph. It is a minor of $G$ and thus has no $K_5$ minor. If $A$ has any vertices, then observe it has at least $5$ such vertices. As $A$ is a bipartite graph and $K_{5}$-minor-free, by Theorem \ref{bipartiteK5bound} we have $e(A)\leq 3v(A)-9$. For every $k\geq 4$, let $n_k$ be the number of elements of $\mathcal{H}$ of type $k$. Then we have $e(A)\geq \sum_{k\geq 4}kn_k\geq 4\sum_{k\geq 4} n_k$ and $v(A)\leq \abs{X}+\sum_{k\geq 4} n_k$. Thus $4\sum_{k\geq 3}n_k\leq 3\abs{X}-9+3\sum_{k\geq 3} n_k$ and so $\sum_{k\geq 3}n_k\leq 3\abs{X}-9$. Hence $v(A)\leq 4\abs{X}-9$. Applying the bound on the number of edges again, we have $\sum_{k\geq 3}kn_k\leq 12\abs{X}-36$. Putting this all together, the total number of vertices in such pockets is at most $(d+1)\sum_{k\geq 4}kn_k<12(d+1)\epsilon v(G)$.

Now we deal with the hardest case, Type 3 pockets. Construct $A$ as in the previous paragraph, but only using pockets of Type 3. That is, contract each pocket of Type 3 down to a single vertex, delete all edges not incident to these vertices and then delete all isolated vertice. Let $h$ be a vertex of $A$ obtained by the contraction of a pocket. It has three neighbours, which we will call $a$, $b$ and $c$ such that there is no rooted $K_{3}$ minor of $A-h$ with roots $a$, $b$ and $c$; that is, a $K_{3}$ minor of $A-h$ such that $a,b,c$ lie in different contracted sets. Now we claim that $A$ has at most $2v(A)-3$ edges. In fact, we prove the following inductive claim.

\begin{claim}
	Let $M$ be a bipartite graph with bipartition $A,B$ such that every vertex in $A$ has degree at most 3. Suppose further that for all $v\in A$ and every cycle $Q$ contained in $M-v$, there is a 1 or 2-vertex-cut between $N(v)$ and $Q$. Then $M$ has at most $2v(M)-4$ edges.
\end{claim}
\begin{claimproof} Suppose not. Let $M$ be a minimum counterexample. Observe $M$ is connected since otherwise if $M$ has components $C_{1},\ldots,C_{t}$, by minimality $e(C_{i}) \leq 2v(C_{i})-4$, and thus it follows that $e(M) \leq 2v(M)-4$, as desired.  Now let $v \in A$. If $M-v$ is acyclic, then $M-v$ has at most $v(M)-2$ edges. Thus $e(M)\leq v(M)+1$ once we have added back in the edges incident to $v$. This is at most $2v(M)-4$ for $v(M)\geq 5$. For $3\leq v(M)\leq 4$, all bipartite graphs have at most $2v(M)-4$ edges.
	
	Thus suppose $M-v$ has a cycle $Q$. There is a 1 or 2-vertex-cut $\{a,b\}$ between $N(v)$ and $Q$. If possible, choose a 1-cut. Let $(X,Y)$ be the separation.  Both $M[X]$ and $M[Y]$ satisfy the assumptions in the claim. Thus $e(X)\leq 2v(X)-4$ and $e(Y)\leq 2v(Y)-4$.
	
	If $a=b$, then $v(M)= v(Y)+v(X)-1$ since the intersection of $X$ and $Y$ is just $a$. Furthermore, $e(M)= e(X)+e(Y)$ since every edge in $M$ is in either $M[X]$ or $M[Y]$. Putting this together, we have that \[e(M)= e(X)+e(Y)\leq 2v(X)-4+2v(Y)-4= 2\left(v(X)+v(Y)-1\right)-6<2v(M)-4.\]
	
	Therefore we may assume that $a \neq b$. We deal with cases depending on where $a$ and $b$ are in the bipartition. 
	If $a$ and $b$ lie in different parts of the bipartition, we construct $X'$ from $X$ and $Y'$ from $Y$ by adding an edge between $a$ and $b$ if it is not already present in $M$.
	As we could not take a single vertex cut, it follows that both $a$ and $b$ had at least one neighbour in $M-X$ and one in $M-Y$. Thus the graphs $X'$ and $Y'$ satisfy the degree conditions of the claim.  Furthermore, one can construct $X'$ (respectively $Y'$) by contracting the edges of a path from $a$ to $b$ in $Y$ (such a path exists or we can find a single vertex cut). Thus any new cycles created in $X'$ or $Y'$ correspond to cycles which were already in $M$. Similarly, 2-vertex cuts are preserved. Therefore in fact we can apply minimality to $X'$ and $Y'$. Now $v(M)=v(X')+v(Y')-2$ and $e(M)\leq e(X')+e(Y')-1$. Thus we have that
	\[e(M)\leq e(X')+e(Y')-1\leq 2v(X')-4+2v(Y')-4-1= 2\left(v(X')+v(Y')-2\right)-5<2v(M)-4.\]
	
	If $a\neq b$ and both are in $A$, we modify the cut as follows. By the assumptions, $a$ has at most 3 neighbours, and since we could not take a single vertex cut, without loss of generality $a$ has a unique neighbour $a'$ in $X$. Then $\{a',b\}$ is also a 2-vertex-cut. This two vertex cut is non-trivial since one side of the separation generated by $\{a,b\}$ contains a cycle (hence at least 4 vertices) and one side contains all the neighbours of $v\in A$. Thus we may instead consider the two vertex cut $\{a',b\}$, and since $a' \in B$, we may apply the argument from the previous paragraph to finish this case. 
	
	If $a\neq b$ and both are in part $B$, we replace $a$ by a path of length 2: $x,a',y$. We give $y$ all the neighbours of $a$ in $Y$ and $x$ all the neighbours of $a$ in $X$. Call the resulting graph $M'$. The only vertex we are adding to $A$ is $a'$. It has degree 2 and a 2-vertex cut $\{x,y\}$ between it and the rest of the graph. Thus $M'$ satisfies the conditions of the claim. Furthermore, we may replace $\{a,b\}$ with $\{a',b\}$ as our 2-cut, thus we may apply the argument to $M'$ where $a'$ and $b$ are in different sides of the bipartition. Note that $v(M')=v(M)+2$ and $e(M')=e(M)+2$. Thus if we want $e(M)\leq 2v(M)-4$, we need $e(M')\leq 2v(M')-6$. To obtain this, we will need to use the additional fact that there is no edge in $M'$ between $a'$ and $b$. Thus it had to be added to $X'$ and $Y'$ (the graphs constructed in that case). Hence $e(M')=e(X')+e(Y')-2$ and $v(M')=v(X')+v(Y')-2$. Thus, as required we have
	\[e(M')= e(X')+e(Y')-2\leq 2v(X')-4+2v(Y')-4-2= 2\left(v(X')+v(Y')-2\right)-6=2v(M')-6.\]
	
	Thus, in all cases, $M$ was not a counterexample to begin with. This is a contradiction so the claim holds.
\end{claimproof}

\vskip.1in

Now returning to our main proof, let $n$ be the number of pockets in $\mathcal{H}$ of type 3. The graph $A$ as constructed above has at most $\abs{X}+n$ vertices and precisely $3n$ edges. Thus $3n\leq 2\abs{X}+2n-4$ and so $n\leq 2\abs{X}-4<2\epsilon v(G)$. Each Type 3 pocket has size at most $3(d+1)$. Thus the number of vertices in such pockets is at most $6(d+1)\epsilon v(G)$.

Putting all this together, at most $18(d+1)\epsilon v(G)$ vertices lie in pockets of $\mathcal{H}$ which cannot belong to a $(C,5,d)$-wallet. As $|X| \leq \epsilon v(G)$, we have a total of $(12d + 13)\epsilon v(G)$ vertices in $G$ which are either in $X$ or pockets that are not suitable for a $(C,5,d)$-wallet.  Let $\mathcal{H}'$ be the subset of $\mathcal{H}$ which includes all pockets that are not Type $2$, $3$ or $k$ for $k \geq 4$. By our choice of $\epsilon$, after the removal of $X$ and all pockets not suitable for a $(C,5,d)$-wallet, there are still at least $\frac{v(G)}{2}$ vertices left over. Since each pocket of $\mathcal{H}'$ has size at most $C$, we have at least $\frac{1}{2C}v(G)$ pockets in $\mathcal{H}'$, and thus a $(C,5,d)$-wallet in $G$, as desired. 
\end{proof}

\subsection{$(C,13k,6k)$-Wallets in $K_t$-Minor-Free Graphs}
In this subsection, we prove the existence of wallets for general $t$. We need the following theorem of the second author and Norin~\cite{norinpostlelist}:

\begin{thm}[Theorem 3.2, \cite{norinpostlelist}]
\label{bipartiteedgeboundlarget}
There exists $C' > 0$ such that for every $t \geq 3$ and every bipartite graph $G$ with bipartition $(A, B)$ and no $K_t$ minor we have
\[e(G) \leq C't\sqrt{\log t}\sqrt{|A||B|} + (t - 2)v(G)\] 
\end{thm}

\begin{thm}
Let $k = \max \{h(t), \lceil \frac{31}{2}(t-1) \rceil\}$, where as before $h(t)$ is the smallest value such that all $K_{t}$-minor-free graphs are $h(t)$-list-colourable. There exists a positive integer $t_{0}$ such that for all $t \geq t_{0}$,
every $K_{t}$-minor-free graph $G$ contains a $(C,13k,6k)$-wallet for some constant $C$.  
\end{thm}

\begin{proof}

Let $C$ be a constant given from Lemma \ref{initialPockets} when applied to the family of $K_{t}$-minor-free graphs. Let $\epsilon$ be a sufficiently small positive value to be determined later. Let $G$ be a $K_{t}$-minor-free graph, and $X \subseteq V(G)$ such that every component of $G-X$ is a $C$-pocket, and $|X| \leq \epsilon v(G)$. As before, let $\mathcal{H}$ be the collection of $C$-pockets in $G-X$.   Let $13k$ be the list sizes from Theorem \ref{K_tDeletability} and let $6k$ be the multiplicative constant from the same theorem. First of all, by definition of $h(t)$, every $K_t$-minor-free graph is $h(t)$-list-colourable. Thus, any pocket with coboundary at most $12h(t)$ is $13h(t)$-deletable.

For ease, let us say a bad pocket of $\mathcal{H}$ is a pocket which cannot belong in a $(C,13k,6k)$-wallet. As before, our goal is to bound the number of vertices in bad pockets. 

Consider the graph $A$ obtained by contracting all bad pockets down to a single vertex,  deleting all of the other pockets, and deleting the edges in $G[X]$. Thus $A$ is bipartite, and as it is a minor of $G$, $K_{t}$-minor-free. Then $X,Y$ is a bipartition of $A$ where $Y$ is the set of all vertices obtained by contracting bad pockets. By Theorem \ref{bipartiteedgeboundlarget}, we have 
\[e(A)\leq C't\sqrt{\log t}\sqrt{\abs{X}\abs{Y}}+(t-2)v(A).\]
Now, we note that every vertex in $Y$ has degree at least $12h(t)$. Also, we have picked $X$ such that $\abs{X}\leq \epsilon v(G)$.

Using these facts, we obtain:
\[12h(t)\abs{Y}\leq C't\sqrt{\log t}\sqrt{\epsilon v(G)\abs{Y}}+(t-2)(\epsilon v(G)+\abs{Y}).\]
Rearraging we get
\[(12h(t)-(t-2))\abs{Y}\leq C't\sqrt{\log t}\sqrt{\epsilon \cdot v(G)\abs{Y}}+(t-2)\cdot \epsilon \cdot v(G).\]
Finally using that $h(t) \geq t$, and using that $p + q \leq 2\max\{p,q\}$, we can simplify this to
\[11h(t)\abs{Y}\leq 2\max\{C't\sqrt{\log t}\sqrt{\epsilon \cdot v(G)\abs{Y}},(t-2)\cdot \epsilon \cdot v(G)\}.\]

Separating this into two inequalities, either we have
\[5.5h(t)\abs{Y}\leq (t-2)\cdot \epsilon \cdot v(G).\]
which simplifies to 
\[\abs{Y}\leq \frac{(t-2)}{5.5h(t)}\cdot \epsilon \cdot v(G),\]
or we have 
\[11h(t)\abs{Y}\leq 2C't\sqrt{\log t}\sqrt{\epsilon \cdot v(G)\abs{Y}}\]
which by dividing by $11h(t)\sqrt{|Y|}$ simplifies to 
\[\sqrt{\abs{Y}}\leq \frac{2C't\sqrt{\log t}\sqrt{\epsilon \cdot v(G)}}{11h(t)},\]
which using the loose inequality $(\frac{a}{b})^{p} \leq \frac{a^{p}}{b}$ gives:
\[\abs{Y}\leq \frac{4C'^2t^2\log t}{11h(t)}\cdot \epsilon \cdot v(G).\]

Thus it follows there exists a $t_{0}$ such that for all $t \geq t_{0}$ we have 
\[\abs{Y}\leq \frac{4C'^2t^2\log t}{11h(t)}\cdot \epsilon \cdot v(G).\]

 Now using this inequality and the inequality given by Theorem \ref{bipartiteedgeboundlarget}  we have that:
\begin{align*}
    e(A)&\leq C't\sqrt{\log t}\sqrt{\epsilon \cdot v(G)\cdot \frac{4C'^2t^2\log t}{11h(t)}\cdot \epsilon\cdot v(G)}+(t-2)\left(\epsilon v(G)+\frac{4C'^2t^2\log t}{11h(t)}\epsilon \cdot v(G)\right)\\
    &\leq C't\sqrt{\log t}\cdot 2C't \cdot \epsilon \cdot v(G)\sqrt{\frac{\log t}{11h(t)}}+(t-2)\frac{4C'^2t^2\log t+11h(t)}{11h(t)}\cdot \epsilon \cdot v(G)\\
    &\leq \epsilon \cdot v(G)\left(\frac{2C'^2t^2\log t}{\sqrt{11h(t)}}+(t-2)\frac{4C'^2t^2\log t+ 11h(t)}{11h(t)}\right).
\end{align*}
We will call the whole expression between parentheses $\kappa(t)$. It is independent of the size of the graph and depends only on $t$. This bound on the number of edges in $A$ implies a bound on the number of vertices in bad pockets since these are shallow. In particular, there are at most $6k\epsilon \cdot \kappa(t) \cdot v(G)$ vertices in bad pockets. If we choose $\epsilon$ small enough that $\epsilon (6k\cdot \kappa(t)+1)< \frac{1}{2}$, then there will still be at least $\frac{v(G)}{2}$ vertices in good pockets. Since each pocket has size at most $C$, this means that there will be at least $\frac{v(G)}{2C}$ good pockets in the collection from $\mathcal{H}$ after removing all bad pockets, and thus this is a wallet.
\end{proof}

\subsection{The Local Algorithm}

In this section we give our algorithm and prove Theorem \ref{algorithmicresults}. We first observe the above results imply the existence of deletable wallets.

\begin{lemma}\label{lem:manageableExist}
 Suppose $t \in \{4,5\}$ or $t$ is sufficiently large and let $G$ be a $K_t$-minor-free graph. If $t = 4$, let $c = 4$. If $t = 5$, let $c=5$. Otherwise, let $c = 13k = 13\max\{h(t),\lceil \frac{32}{2}(t-1)\rceil \}$.	There is a constant $d$ such that the following holds for all $C$. If $P$ is a $d$-deep $C$-pocket of $G$ with coboundary $X$, then there is a non-empty $c$-deletable subgraph $H$ of $P$.
\end{lemma}

\begin{proof}
If $t=4$, set $d=12$. If $t=5$, set $d$ to be the constant in Theorem \ref{K5minorfree}. If $t$ is large, let $d = 6k$. Let $P$ be a $d$-deep $C$-pocket of $G$ with coboundary $X$, and further suppose that $P$ is a vertex minimum counterexample. Note $P$ is not $c$-deletable, as otherwise $P$ would not be a counterexample. Thus $G[P \cup X]$ has no deletable subgraph disjoint from $X$, and therefore by Theorems \ref{K4minorfree}, \ref{K5minorfree}, and \ref{largevaluesoft}, $v(P)+\abs{X}\leq d\abs{X}$. Thus $v(P)\leq (d-1)\abs{X}$. This contradicts the assumption that $P$ is $d$-deep, giving the result.
\end{proof}

Putting everything together, we get the following theorem:
\begin{thm}
\label{thm:manageableExist}
Every $K_{4}$-minor-free graph contains a deletable $(C_{1},4)$-wallet for some constant $C_{1}$. There exists constants $C_{2}$ such that every $K_{5}$-minor-free graph has a deletable $(C_{2},5)$-wallet. 
For $k = 13\max\{h(t),\lceil \frac{32}{2}(t-1)\rceil \}$, and sufficiently large $t$, there exists a constant $C_{3}$ such that every $K_{t}$-minor-free graph contains a deletable $(C_{3},13k)$-wallet. 
\end{thm}

\begin{proof}
By Theorem \ref{K4ClutchExist}, $K_{4}$-minor-free graphs have $(C,4,12)$-wallets. Let $\mathcal{H}$ be any such wallet. By applying Lemma \label{lem:manageableExist}, to all components of $\mathcal{H}$ that are not $c$-deletable, we obtain a set of components $\mathcal{H'}$ where each component is $c$-deletable, a $C$-pocket, all components are non-touching and $|\mathcal{H}'| = |\mathcal{H}|$. It easily follows there exists a constant $C_{1}$ such that $\mathcal{H}'$ is a deletable $(C_{1},4,12)$-wallet.

The same argument works for $K_{5}$-minor-free graphs by appealing to Theorem \ref{K5ClutchExist} instead of Theorem \ref{K4ClutchExist}, and similarly for the large $K_{t}$-minor-free case. 
\end{proof}

The following theorem follows from the proofs given in \cite{Postlealgorithms}: 

\begin{thm}[\cite{Postlealgorithms}]\label{thm:postlealgorithm}
If $\mathcal{G}$ is any family of graphs satisfying the following three conditions:
\begin{itemize}
	\item $\mathcal{G}$ is closed under taking subgraphs,
	\item for every $G\in\mathcal{G}$ with $v(G)>k$,  the graph $G$ contains a deletable $(C,c,d)$-wallet for some fixed constants $C,c,$ and $d$, 
	\item every $G\in\mathcal{G}$ is $c$-list-colourable, 
\end{itemize}
then Algorithm \ref{ColourAlg} returns a $c$-list-colouring of $G$ in $O(\log v(G))$ rounds. 
\end{thm}

We note that our definition of wallet is slightly different the the one given in \cite{Postlealgorithms}, but nevertheless the same argument goes through with our definition. Therefore Theorem \ref{algorithmicresults} follows immediately. 

\begin{algorithm}\caption{Distributed Colouring algorithm}\label{ColourAlg}
	\begin{algorithmic}[1]
		\Params
			\begin{itemize}
				\item Constant $C\geq 2$
				\item Constant $k$
			\end{itemize}
		\EndParams
		\Input
			\begin{itemize}
				\item Graph $G$
				\item $c$-List assignment $L$ of $G$
			\end{itemize} 
		\EndInput
		\Output
			An $L$-colouring $\phi$ of $G$.
		\EndOutput
		\If{$v(G)\leq k$} find an $L$-colouring $\phi$ of $G$ by exhaustive search and \Return it.
		\Else 
			\ForAll{$v\in V(G)$ of degree at most $C$}
				Find, if one exists, a $c$-deletable pocket $H_v$ containing $v$.
			\EndFor
			\State Set $H=\bigcup_v H_v$ and let $G_{0} = G-H$
			\Recurse on $G_0$ and $L$ to obtain $\phi_0$ \EndRecurse
			\State In parallel to the recursion, find a $C^{2C}$-colouring $\psi$ of $H^{2C}$.
			\For{$i=1,...,C^{2C}$}
				\begin{itemize}
					\item Let $H_i=\psi^{-1}(\{i\})$
					\item Let $G_i=G[V(G_{i-1})\cup H_i]$
					\item Restrict $\phi_{i-1}$ to $G_i-H_i$ and then extend it to $G_i$ to obtain $\phi_i$
				\end{itemize}
			\EndFor
		\EndIf
		\State \Return $\phi$
	\end{algorithmic}
\end{algorithm}

Here is a sketch of the proof of Theorem~\ref{thm:postlealgorithm}.

\begin{proof}[Sketch of proof]
    We now sketch the proof. Let $G\in \mathcal{G}$ with $n$ vertices. In step 6, we find a deletable pocket for every vertex that lies in such. In step 10, we use colouring as a way to find sets of non-touching pockets. Since there is a $(C,c,d)$-wallet, we are guaranteed to find a large (linear in $n$) set of non-touching deletable pockets. We set these aside. Since $\mathcal{G}$ is closed under taking subgraphs, once this wallet is removed, the remaining graph is still in $\mathcal{G}$ and so has a deletable wallet. At each step, we set aside at least a constant fraction of the remaining vertices. Thus we need only run the recursion a logarithmic number of times. Furthermore, once we have coloured everything apart from what was set aside before recursion $i$, we can extend our colouring back to what was set aside at recursion $i$ because it was deletable with respect to what was left when we set it aside.
\end{proof}

\section{Lower Bounds}
\label{lowerboundssection}
This section is dedicated to proving the following observation.

\begin{obs}
For each integer $t\ge 3$ and positive integer $n$, there exists a graph $G$ on $n(t-1)+1$ vertices which is $\lceil n/2 \rceil$-locally-$K_{t}$-minor-free but not $(t-1)$-colourable.
\end{obs}

\begin{proof}
Let $H_t$ be the graph obtained from $K_t$ by removing one edge. Note $K_t$ has exactly two vertices of degree $t-2$. Call them $a$ and $b$. Let $n\geq 3$ be a given integer. Let $G_t^n$ be the graph obtained by taking $n$ copies of $H_t$ and placing them in an order, identifying the vertex $a$ of every copy of $H_t$ with the vertex $b$ of the next and then finally adding an edge between the $b$ vertex of the last $H_t$ and the $a$ vertex of the first.

We claim that for every positive integer $n$, $G_t^n$ is not $(t-1)$-colourable. Suppose it were. Consider such a colouring on each copy of $H_t$. There are $t-1$ colours and $t$ vertices so at least one colour must be reused. The only pair of vertices in $H_t$ that are not adjacent are $a$ and $b$. Thus these two must share a colour while the other $t-2$ vertices get their own colour. Now, in the construction of $G_t^n$, we identified the $a$ and $b$ vertices of consecutive copies of $H_t$. Thus all $a$ and $b$ vertices in all the copies of $H_t$ must share the same colour. However, there is an edge between the $a$ vertex of the first copy of $H_t$ and the $b$ vertex of the last copy of $H_t$. Thus both endpoints of this edge share the same colour and we have a contradiction. Thus $G_t^n$ is not $(t-1)$-colourable. Also note that, $v(G_t^n)=n(t-1)+1\geq n(t-1)$.

Next, we claim that $G_t^n$ is $\lfloor n/2 \rfloor $-locally-$K_t$-minor-free. Suppose not. Let $B$ be a neighbourhood of $G_t^n$ of radius $n/2$ which has a $K_t$ minor. Then there must be connected subgraphs $B_1,...,B_t$ of $B$ such that for every pair $1\leq i<j\leq t$, there is an edge with one endpoint in $B_i$ and the other in $B_j$. By construction, $B$ is not two connected since it does not have sufficient radius to contain the full necklace structure. Its block decomposition consists of up to $n/4+1$ copies of $H_t$, up to two copies of $K_{t-1}$ and up to one copy of $K_2$. The single vertex shared by any pair of blocks may only lie in up to one $B_i$. Thus there must be a block containing a vertex in every $B_i$. Otherwise, there could not be an edge between subgraphs that are in different blocks but do not contain the shared vertex. The only blocks with at least $t$ vertices are the copies of $H_t$. However, these do not have a $K_t$ minor. Thus $B$ has no $K_t$ minor.

\end{proof}




\subsection*{Acknowledgments}

The authors would like to thank Daniel Cranston, Michelle Delcourt, Louis Esperet, Matthew Kroeker and Ronen Wdowinski for helpful discussions related to the presentation of this paper.

\bibliographystyle{plain}  
\bibliography{HadwigerBib}

\end{document}